\documentclass[11pt,letterpaper]{article}

\usepackage[T1]{fontenc}
\usepackage[utf8]{inputenc}
\usepackage{lmodern}
\usepackage{amsmath,amssymb,amsthm,mathtools}
\usepackage[margin=0.72in,includefoot]{geometry}
\usepackage{microtype}
\usepackage{array}
\usepackage{authblk}
\usepackage[hidelinks]{hyperref}

\allowdisplaybreaks
\newtheorem{theorem}{Theorem}[section]
\newtheorem{lemma}[theorem]{Lemma}

\newcommand{\ex}{\operatorname{ex}}
\newcommand{\N}{\operatorname{N}}
\newcommand{\ind}{\mathbf{1}}

\title{The exact generalized Tur\'an number for \(C_6\) in \(C_8\)-free graphs}
\author[1]{Zian Chen\thanks{Email: \texttt{michaelchen24@163.com}}}
\author[1]{Jinghua Deng\thanks{Email: \texttt{Jinghua\_deng@163.com}}}
\affil[1]{Center for Discrete Mathematics,
            Fuzhou University, Fujian, 350003, China}
\date{}

\begin{document}

\maketitle

\begin{center}
  \begin{minipage}{0.88\textwidth}
    \begin{center}
      \textbf{Abstract}
    \end{center}
    For graphs $F$ and $H$, let $\ex(n,F,H)$ denote the maximum number of copies of $F$ in an $n$-vertex $H$-free graph. Gerbner, Gy\H{o}ri, Methuku and Vizer proved that $\ex(n,C_6,C_8)=\Theta(n^3)$ and predicted that the unrestricted problem should have the same first-order asymptotics as the bipartite one. We determine the exact value for all sufficiently large $n$, showing that
\[
  \ex(n,C_6,C_8)=6\binom{n-3}{3}+12(n-5).
\]
Moreover, the unique extremal graph is $K_3\vee (K_2\cup I_{n-5})$. The main new ingredient is a codegree decomposition for $C_8$-free graphs: a packing lemma for triangles in the linear-codegree graph recovers an almost spanning common neighborhood, and a defect-absorption argument upgrades this stability to the exact extremal graph.
  \end{minipage}
\end{center}

\medskip
\noindent\textbf{Keywords.} generalized Tur\'an number; even cycle; codegree; stability; exact extremal graph.

\noindent\textbf{2020 Mathematics Subject Classification.} 05C35, 05C38.

\section{Introduction}

Let $H$ be a fixed graph. A graph is $H$-free if it contains no copy of $H$ as a subgraph, and the classical Tur\'an number $\ex(n,H)$ is the maximum number of edges in an $n$-vertex $H$-free graph. Determining $\ex(n,H)$ is one of the central problems of extremal graph theory. The Erd\H{o}s--Stone--Simonovits theorem~\cite{ErdosStone1946,ErdosSimonovits1966} gives the asymptotic answer when $H$ is non-bipartite, but the bipartite case remains much more delicate. 

Among bipartite forbidden graphs, even cycles form a basic sparse test family: Bondy and Simonovits proved that every $C_{2k}$-free $n$-vertex graph has $O_k(n^{1+1/k})$ edges, and later refinements sharpen this phenomenon in several directions \cite{BondySimonovits1974,BukhJiang2017}. In particular, a $C_8$-free graph has only $O(n^{5/4})$ edges, so density alone is too crude for sharp cycle-counting questions; the decisive local quantities are usually common neighborhoods, or equivalently pair-codegrees.

This paper studies a generalized Tur\'an problem in this sparse even-cycle setting. For graphs $F$ and $H$, let $\ex(n,F,H)$ denote the maximum number of copies of $F$ in an $n$-vertex $H$-free graph; copies are not required to be induced. The ordinary Tur\'an number is the special case $F=K_2$. This viewpoint was systematized by Alon and Shikhelman \cite{AlonShikhelman2016}, building on earlier clique-counting results of Zykov~\cite{Zykov1949} and Erd\H{o}s~\cite{Erdos1962} and on further exact results such as the theorem of Gy\H{o}ri, Pach and Simonovits \cite{GyoriPachSimonovits1991} for many bipartite graphs in triangle-free hosts. For a recent survey of the area, see Gerbner and Palmer \cite{GerbnerPalmerSurvey2026}. Here we focus on the cycle-versus-cycle instance
\[
  \ex(n,C_6,C_8).
\]

For generalized Tur\'an problems between cycles, the known theory separates naturally according to the parity of the forbidden cycle. If the forbidden cycle is odd, the host graph may still be dense, and sharp results often rely on dense extremal and stability methods; for example, the maximum number of $C_5$ copies in a triangle-free graph was determined independently by Grzesik~\cite{Grzesik2012} and by Hatami, Hladk\'y, Kr\'al', Norine and Razborov~\cite{HatamiEtAl2013}, while related problems on triangles in graphs with no long odd cycle were studied by Bollob\'as and Gy\H{o}ri~\cite{BollobasGyori2008} and by Gy\H{o}ri and Li \cite{GyoriLi2012}. When the forbidden cycle is even, the host graph is sparse and codegrees become the natural substitute for edge density. At the level of exponents, Gishboliner and Shapira~\cite{GishbolinerShapira2020} gave a general classification for pairs of cycles, and Gerbner, Gy\H{o}ri, Methuku and Vizer \cite{GerbnerGyoriMethukuVizer2020} independently developed the even-cycle case. In particular, for distinct $k,\ell\ge 2$,
\[
  \ex(n,C_{2\ell},C_{2k})=\Theta(n^\ell),
\]
and hence $\ex(n,C_6,C_8)=\Theta(n^3)$. This determines the order of magnitude, but it leaves open the leading constant and the extremal configuration.

A natural next problem is therefore to determine the exact asymptotics, and ultimately the exact value, for the first unresolved sparse case $C_6$ versus $C_8$. Let $\ex_{bip}(n,C_6,C_8)$ denote the maximum number of copies of $C_6$ in a $C_8$-free bipartite graph. Gerbner, Gy\H{o}ri, Methuku and Vizer singled out this pair and proved for bipartite host graphs that
\[
  \ex_{bip}(n,C_6,C_8)=n^3+O(n^{5/2}),
\]
while asking for the unrestricted asymptotics and predicting that the first-order term should be the same \cite{GerbnerGyoriMethukuVizer2020}. The bipartite problem was later solved exactly by Gy\H{o}ri, He, Lv, Salia, Tompkins, Varga and Zhu: for every $n\ge 6$,
\[
  \ex_{bip}(n,C_6,C_8)=6\binom{n-3}{3},
\]
with equality only for $K_{3,n-3}$ \cite{GyoriHeLvSaliaTompkinsVargaZhu2025}. The same paper also determines, for every fixed $s\ge 2$ and all sufficiently large $n$, the value of
\[
  \ex(n,K_{s,s},C_{2s+2}).
\]
In the case $s=3$, this result points to the same $K_{3,n-3}$-type core as the relevant extremal configuration. However, counting $K_{3,3}$ copies does not determine the number of $C_6$ copies: each $K_{3,3}$ contains six such cycles, while internal edges may create additional $C_6$ copies not accounted for by the $K_{3,3}$-count.


Our main contribution is to convert these first-order and structural indications into an exact theorem. Let $I_t$ denote the edgeless graph on $t$ vertices, and let $G\vee H$ denote the join of two vertex-disjoint graphs. For $n\ge 5$, define
\[
  H_n:=K_3\vee (K_2\cup I_{n-5}).
\]

\begin{theorem}\label{thm:main}
For all sufficiently large $n$,
\[
  \ex(n,C_6,C_8)=6\binom{n-3}{3}+12(n-5).
\]
Moreover, up to isomorphism, the unique extremal graph is $K_3\vee (K_2\cup I_{n-5})$.
\end{theorem}

Let $\N(F,H)$ denote the number of $F$-copies in $H$. It is clear that $H_n$ is $C_8$-free and 
 a direct count gives
\[
 \N(C_6,H_n)=6\binom{n-3}{3}+12(n-5).
\]

The rest of the paper is organized as follows. Section~2 develops the codegree decomposition, bounds the contributions of triples with too few large-codegree pairs, proves the packing estimate for fat triangles, and derives the stability statement giving an almost spanning common neighborhood of three vertices. Section~3 removes the remaining exceptional vertices by a defect-absorption argument and then performs the terminal count, completing the proof of Theorem~\ref{thm:main}.

\section{The codegree decomposition}

Throughout this section, \(G\) is an \(n\)-vertex \(C_8\)-free graph. For distinct vertices \(u,v\), write
\[
  N(u,v):=N(u)\cap N(v),\qquad d(u,v):=|N(u,v)|.
\]
We let \(\N(C_\ell,G)\) denote the number of unoriented copies of \(C_\ell\) in \(G\), with additional
chords ignored. For an unordered pair \(e=uv\), write \(d(e):=d(u,v)\). Constants implicit in
\(O(\cdot)\) may depend on displayed subscripts.

The following facts will be used repeatedly. The first is the even-cycle theorem of Bondy and
Simonovits.

\begin{theorem}\label{thm:bondy-simonovits}~\cite{BondySimonovits1974}
For every integer \(k\ge2\), every \(n\)-vertex \(C_{2k}\)-free graph has \(O_k(n^{1+1/k})\)
edges. In particular, every \(n\)-vertex \(C_8\)-free graph has \(O(n^{5/4})\) edges.
\end{theorem}

The second is due to Gerbner, Gy\H{o}ri, Methuku and Vizer.

\begin{theorem}\label{thm:ggmv-c4}~\cite{GerbnerGyoriMethukuVizer2020}
Every \(n\)-vertex \(C_8\)-free graph contains \(O(n^2)\) copies of \(C_4\).
\end{theorem}

\begin{lemma}\label{lem:basic}
  Every \(n\)-vertex \(C_8\)-free graph satisfies
  \[
    e(G)=O(n^{5/4}),\qquad
    \sum_{u<v}d(u,v)^2=O(n^2),\qquad
    \N(C_6,G)=O(n^3).
  \]
  For \(L>0\), let
  \[
    \tilde{G}_L:=\{uv:d(u,v)\ge L\}.
  \]
  There is an absolute \(K_0\) such that \(\tilde{G}_K\) is \(C_4\)-free for \(K\ge K_0\). For every fixed
  \(\rho>0\), the graph \(\tilde{G}_{\rho n}\) is \(C_4\)-free for all sufficiently large \(n\), and
  \[
    e(\tilde{G}_{\rho n})=O_\rho(1).
  \]
  Finally, four fixed distinct vertices have at most three common neighbors.
\end{lemma}

\begin{proof}
  By Theorems~\ref{thm:bondy-simonovits} and~\ref{thm:ggmv-c4}, we have
  \(e(G)=O(n^{5/4})\) and \(\N(C_4,G)=O(n^2)\). Hence
  \[
    \sum_{u<v}\binom{d(u,v)}{2}=2\N(C_4,G)=O(n^2).
  \]
  For every integer \(d\ge0\), $d^2\le4\binom d2+\ind_{\{d=1\}}$.
  Consequently,
  \begin{align}\label{eq:sum-codegree-O(n2)}
      \sum_{u<v}d(u,v)^2
    \le4\sum_{u<v}\binom{d(u,v)}2+\binom n2
    =O(n^2).
  \end{align}

  Let \(B\) be the symmetric matrix with \(B_{uv}=d(u,v)\) for \(u\ne v\) and zero diagonal. The sum
    $\sum_{a,b,c}B_{ab}B_{bc}B_{ca}$ counts ordered sextuples \((a,x,b,y,c,z)\) such that $x\in N(a,b),y\in N(b,c),z\in N(c,a)$, 
  or equivalently alternating closed walks $a$-$x$-$b$-$y$-$c$-$z$-$a$. Each unoriented copy of \(C_6\) gives
  exactly \(6\cdot2=12\) such sextuples, one for each starting vertex and direction around the cycle.
  All remaining sextuples only add nonnegative terms, possibly degenerate. Hence
  \[
    12\N(C_6,G)\le \sum_{a,b,c}B_{ab}B_{bc}B_{ca}=\operatorname{tr}(B^3).
  \]
  Let
  \(\lambda_1,\ldots,\lambda_n\) be the eigenvalues of \(B\). Since \(B\) is symmetric,
  \[
    \operatorname{tr}(B^3)=\sum_i\lambda_i^3
    \le\sum_i|\lambda_i|^3
    \le\left(\sum_i\lambda_i^2\right)^{3/2}
    =\left(\sum_{u,v}B_{uv}^2\right)^{3/2}
    =O(n^3).
  \]

  Suppose first that \(v_1v_2v_3v_4v_1\) is a 4-cycle in \(\tilde{G}_K\). If \(K\ge8\), we may choose
  the vertices \(w_i\in N(v_i,v_{i+1})\) successively so that
  \[
    w_i\notin\{v_1,v_2,v_3,v_4,w_1,\ldots,w_{i-1}\}.
  \]
  Indeed, when \(w_i\) is chosen, at most seven vertices are forbidden and
  \(|N(v_i,v_{i+1})|\ge K\ge8\). They form the cycle
  $v_1$-$w_1$-$v_2$-$w_2$-$v_3$-$w_3$-$v_4$-$w_4$-$v_1$,
  a contradiction. Thus \(\tilde{G}_K\) is \(C_4\)-free whenever \(K\ge8\), so one may take \(K_0=8\).
  The same greedy choice applies to \(\tilde{G}_{\rho n}\) once \(n\) is sufficiently large in terms of \(\rho\).
  By~\eqref{eq:sum-codegree-O(n2)}, we have
  \begin{align*}
    e(\tilde{G}_{\rho n})\rho^2n^2
    \le\sum_{uv\in E(\tilde{G}_{\rho n})}d(u,v)^2
    \le\sum_{u<v}d(u,v)^2
    =O(n^2),
  \end{align*}
  and thus \(e(\tilde{G}_{\rho n})=O_\rho(1)\).
  For the last assertion, if distinct vertices \(x_1,x_2,x_3,x_4\) had distinct common neighbors
  \(y_1,y_2,y_3,y_4\), then $x_1$-$y_1$-$x_2$-$y_2$-$x_3$-$y_3$-$x_4$-$y_4$-$x_1$ would be a \(C_8\).
\end{proof}

Every copy $C=a$-$x$-$b$-$y$-$c$-$z$-$a$ has alternating triples \(T=\{a,b,c\}\) and \(W=\{x,y,z\}\). For a triple
\(S=\{u,v,w\}\), put
\[
  \mu(S):=\frac{d(u,v)+d(v,w)+d(w,u)}{3}.
\]
Fix an arbitrary total order on \(V(G)\), and compare triples lexicographically after listing their
vertices in increasing order.
For any such \(C\), charge \(C\) to \(\{a,b,c\}\) if
\(\mu(\{a,b,c\})>\mu(\{x,y,z\})\), and to \(\{x,y,z\}\) if
\(\mu(\{x,y,z\})>\mu(\{a,b,c\})\). In the case of equality, charge \(C\) to the lexicographically
smaller of the two triples. For any triple \(S\), define
\[
  F^{\#}(S):=
  |\{H\subseteq G:H\cong C_6\text{ and }H\text{ is charged to }S\}|.
\]
Independent choices of the three intermediate vertices and the arithmetic-geometric mean inequality give
\begin{equation}\label{eq:charge}
  F^{\#}(S)\le d(u,v)d(v,w)d(w,u)\le\mu(S)^3.
\end{equation}
Fix \(0<\rho<1\) and write \(\tilde{G}:=\tilde{G}_{\rho n}\). A triple is called \(i\)-fat if exactly \(i\) of its three pairs are edges of \(\tilde{G}\).

\subsection{Fat triangles}

A 3-fat triple is a triangle of \(\tilde{G}\). Write \(\triangle(\tilde{G})\) for the set of triangles of \(\tilde{G}\). For
\(T=\{a,b,c\}\in\triangle(\tilde{G})\), define
\[
  P_T:=N_G(a,b)\cup N_G(b,c)\cup N_G(c,a).
\]

\begin{lemma}\label{lem:packing}
For all sufficiently large \(n\),
\[
  \sum_{T\in\triangle(\tilde{G})}|P_T|\le n+O_\rho(1).
\]
\end{lemma}

\begin{proof}
By Lemma~\ref{lem:basic}, \(\tilde{G}\) is \(C_4\)-free and has \(O_\rho(1)\) edges. Thus it has
\(O_\rho(1)\) triangles, and two of them cannot share an edge. Let 
\[
  V^*(\tilde{G})\coloneqq \{v\in V(G):d_{\tilde{G}}(v)\ge1\}\quad\text{ and }\quad P_T^\circ:=P_T\setminus V^*(\tilde{G}).
\]
Since there are \(O_\rho(1)\) triangles in \(\tilde{G}\), it is enough to show that
\(|P_T^\circ\cap P_{T'}^\circ|=O(1)\) for distinct triangles \(T,T'\).

Take \(x\in P_T^\circ\cap P_{T'}^\circ\). Then \(x\) lies in \(N(f)\cap N(f')\) for some pair
\(f\subset T\) and some pair \(f'\subset T'\). If \(f\) and \(f'\) are disjoint, their four endpoints
have \(x\) as a common neighbor, so there are at most three possibilities for \(x\).

Suppose next that \(f\) and \(f'\) meet. Relabel so that
\[
  T=\{a,b,c\},\qquad T'=\{a,d,h\},\qquad f=ab,\qquad f'=ad.
\]
There is at most one vertex in \(N(a,b)\cap N(a,d)\setminus V^*(\tilde{G})\). Otherwise take two such
vertices \(x_1,x_2\). Since \(dh,ha\in E(\tilde{G})\), the sets \(N_G(d,h)\) and \(N_G(h,a)\)
each have size at least \(\rho n\). For sufficiently large \(n\), choose
\[
  y\in N_G(d,h),\qquad z\in N_G(h,a)
\]
outside \(\{a,b,c,d,h,x_1,x_2\}\) and with \(y\ne z\).
Then $b$-$x_1$-$d$-$y$-$h$-$z$-$a$-$x_2$-$b$ is a \(C_8\). Since there are only nine choices for \((f,f')\), the required intersection bound follows.

Finally,
\[
  \sum_T|P_T^\circ|
  \le\left|\bigcup_TP_T^\circ\right|
     +\sum_{T<T'}|P_T^\circ\cap P_{T'}^\circ|
  \le n+O_\rho(1).
\]
Indeed, an element lying in \(k\) of the sets contributes \(k\) to the left-hand side and
\(1+\binom{k}{2}\ge k \) to the first two terms on the right. Moreover,
\(\left|\bigcup_TP_T^\circ\right|\le n\), and the pairwise intersections contribute \(O_\rho(1)\)
in total because \(\tilde{G}\) has \(O_\rho(1)\) triangles. This proves the displayed bound.
\end{proof}

Since \(\mu(T)\le|P_T|\), equation~\eqref{eq:charge} and Lemma~\ref{lem:packing} give
\begin{equation}\label{eq:fat-triangles}
  \sum_{T\text{ 3-fat}}F^{\#}(T)
  \le\sum_T\mu(T)^3
  \le\left(\sum_T|P_T|\right)^3
  \le n^3+O_\rho(n^2).
\end{equation}

\subsection{One or two fat pairs}

We shall use the following linear budget for the edges of \(\tilde{G}\).

\begin{lemma}\label{lem:linear-budget}
For every fixed \(\rho>0\),
\[
  \sum_{e\in E(\tilde{G})}d_G(e)\le3n+O_\rho(1).
\]
\end{lemma}

\begin{proof}
Let \(V^*(\tilde{G}):=\{v\in V(G):d_{\tilde{G}}(v)\ge1\}\).
For \(x\in V(G)\), let \(\tilde{G}_x:=\tilde{G}[N_G(x)]\). 
Since \(e(\tilde{G})=O_\rho(1)\), also \(|V^*(\tilde{G})|\le2e(\tilde{G})=O_\rho(1)\).

Let \(X:=\{x\in V(G):e(\tilde{G}_x)\ge4\}\). If \(x\in X\), then the four edges in \(\tilde{G}_x\) involve at
least four vertices of \(V^*(\tilde{G})\), all adjacent to \(x\) in \(G\). Thus \(x\) is a common neighbor
of some four-element subset of \(V^*(\tilde{G})\). By Lemma~\ref{lem:basic}, each such four-element
subset has at most three common neighbors. Hence
\[
  |X|\le3\binom{|V^*(\tilde{G})|}{4}=O_\rho(1).
\]
For \(x\notin X\), we have \(e(\tilde{G}_x)\le3\). For \(x\in X\), the trivial bound
\(e(\tilde{G}_x)\le e(\tilde{G})=O_\rho(1)\) applies. Therefore, counting each edge \(uv\in E(\tilde{G})\) once for every
common neighbor \(x\in N_G(u,v)\) gives
\[
  \sum_{e\in E(\tilde{G})}d_G(e)=\sum_{x\in V(G)}e(\tilde{G}_x)\le3n+|X|e(\tilde{G})=3n+O_\rho(1),
\]
which proves the claim.
\end{proof}

\begin{lemma}\label{lem:one-two-fat}
For fixed \(K\ge10\), the total contribution of 2-fat charged triples is
\[
  O(\rho n^3)+O_\rho(n^2),
\]
and the total contribution of 1-fat charged triples is
\[
  O_K(\rho n^3)+O_{K,\rho}(n^2).
\]
\end{lemma}

\begin{proof}
We first treat 2-fat triples. Let \(T=\{v_1,v_2,v_3\}\) be 2-fat, and suppose, after relabeling, that
  \(v_1v_2,v_1v_3\in E(\tilde{G})\), and \(v_2v_3\notin E(\tilde{G})\).

By the definition of \(\tilde{G}\), the last condition gives \(d(v_2,v_3)<\rho n\). Therefore, by
\eqref{eq:charge},
\[
  F^{\#}(T)
  \le d(v_1,v_2)d(v_1,v_3)d(v_2,v_3)
  \le \rho n\,d(v_1,v_2)d(v_1,v_3).
\]
For each 2-fat triple \(T\), let \(e_1(T),e_2(T)\) be its two fat pairs. Summing the previous bound
and then enlarging the sum to all ordered pairs of edges of \(\tilde{G}\), we obtain
\[
  \sum_{T\text{ 2-fat}}F^{\#}(T)
  \le \rho n\sum_{T\text{ 2-fat}}d(e_1(T))d(e_2(T))
  \le \rho n\sum_{e_1,e_2\in E(\tilde{G})}d(e_1)d(e_2)
  =\rho n\left(\sum_{e\in E(\tilde{G})}d(e)\right)^2.
\]
Lemma~\ref{lem:linear-budget} gives
\[
  \sum_{T\text{ 2-fat}}F^{\#}(T)
  \le \rho n(3n+O_\rho(1))^2
  =O(\rho n^3)+O_\rho(n^2).
\]

We next treat 1-fat triples. Fix an edge \(v_1v_2\in E(\tilde{G})\), and for
\(\omega_0\in V(G)\setminus\{v_1,v_2\}\), write
\[
  p_{\omega_0}=d(v_1,\omega_0),\qquad q_{\omega_0}=d(v_2,\omega_0).
\]
We claim that
\begin{equation}\label{eq:pq}
  \sum_{\substack{\omega_0\in V(G)\setminus\{v_1,v_2\}\\p_{\omega_0},q_{\omega_0}<\rho n}}
  p_{\omega_0}q_{\omega_0}
  =O_K(\rho n^2).
\end{equation}
Indeed, vertices \(\omega_0\) with \(\min\{p_{\omega_0},q_{\omega_0}\}<K\) contribute at most \(K\rho n\)
each, since \(p_{\omega_0},q_{\omega_0}<\rho n\). Thus their total contribution is at most
\(K\rho n^2\).

There is at most one vertex \(\omega_0\) with \(p_{\omega_0},q_{\omega_0}\ge K\). Indeed, if two
distinct vertices \(\omega_1,\omega_2\) satisfied
$p_{\omega_1},q_{\omega_1},p_{\omega_2},q_{\omega_2}\ge K$,
then, since \(K\ge10\), we could choose vertices
\[
  \omega_3\in N(v_1,\omega_1),\quad
  \omega_4\in N(\omega_1,v_2),\quad
  \omega_5\in N(v_2,\omega_2),\quad
  \omega_6\in N(\omega_2,v_1)
\]
successively so that they are distinct and avoid the already named vertices
\(\{v_1,v_2,\omega_1,\omega_2\}\). Indeed, at each step at most seven vertices are forbidden,
whereas each relevant common neighborhood has size at least \(K\).
Then
  $v_1$-$\omega_3$-$\omega_1$-$\omega_4$-$v_2$-$\omega_5$-$\omega_2$-$\omega_6$-$v_1$
is a \(C_8\), a contradiction. Hence at most one remaining vertex exists, and its contribution is at
most \((\rho n)^2\le\rho n^2\). This proves \eqref{eq:pq}.

For a 1-fat triple whose unique fat pair is \(v_1v_2\), the third vertex \(\omega_0\) satisfies
\(p_{\omega_0},q_{\omega_0}<\rho n\), and
\[
  F^{\#}(\{v_1,v_2,\omega_0\})\le d(v_1,v_2)p_{\omega_0}q_{\omega_0}.
\]
Therefore, using \eqref{eq:pq} for each fixed fat edge and then Lemma~\ref{lem:linear-budget},
\[
  \sum_{T\text{ 1-fat}}F^{\#}(T)
  \le
  \sum_{v_1v_2\in E(\tilde{G})}d(v_1,v_2)
  \sum_{\substack{\omega_0\in V(G)\setminus\{v_1,v_2\}\\p_{\omega_0},q_{\omega_0}<\rho n}}
  p_{\omega_0}q_{\omega_0}
  \le O_K(\rho n^2)\sum_{e\in E(\tilde{G})}d(e)
  =O_K(\rho n^3)+O_{K,\rho}(n^2).
\]
This proves the 1-fat estimate.
\end{proof}

\subsection{The zero-fat part}

The zero-fat triples are treated by separating uniformly large codegrees from cycles that are thin on
both alternating sides.

\begin{lemma}\label{lem:all-large}
Let \(K\) be sufficiently large and \(L\ge K\). Then
\[
  \sum_{\substack{\{p,q,r\}:\\K\le d(p,q),d(q,r),d(r,p)<L}}
  d(p,q)d(q,r)d(r,p)=O_K(Ln^2).
\]
\end{lemma}

\begin{proof}
If there is no triple \(\{v_1,v_2,v_3\}\) satisfying
 $ K\le d(v_1,v_2),d(v_1,v_3),d(v_2,v_3)<L$,
then the left-hand side is zero. Thus it remains to consider a triple \(\{v_1,v_2,v_3\}\) and
a dyadic value \(M\in\{K,2K,2^2K,\ldots\}\) such that \(M\le L\),
\(K\le d(v_1,v_2),d(v_1,v_3),d(v_2,v_3)<L\), and
\(\max\{d(v_1,v_2),d(v_1,v_3),d(v_2,v_3)\}\in[M,2M)\).

We assume that $d(v_1,v_2)\ge\max\{d(v_2,v_3),d(v_1,v_3)\}$. Then
\[
  v_1v_2,\ v_1v_3,\ v_2v_3\in E(\tilde{G}_K),
  \qquad
  v_1v_2\in E(\tilde{G}_M).
\]

Since \(K\) is sufficiently large, Lemma~\ref{lem:basic} gives that \(\tilde{G}_K\) is \(C_4\)-free.
Hence an edge of \(\tilde{G}_K\) lies in at most one triangle of \(\tilde{G}_K\). Therefore the map from
triples at scale \(M\) to their chosen maximum pair is injective. Hence the number of triples at this
scale is at most \(e(\tilde{G}_M)\).

By Lemma~\ref{lem:basic},
\begin{align*}
   e(\tilde{G}_M)M^2\le\sum_{uv\in E(\tilde{G}_M)}d(u,v)^2=O(n^2)
   \quad\text{ and }\quad
   e(\tilde{G}_M)=O(\frac{n^2}{M^2}).
\end{align*}
 
For every triple at this scale, all three pair-codegrees are below \(2M\), so
\[
  d(v_1,v_2)d(v_1,v_3)d(v_2,v_3)\le (2M)^3=O(M^3).
\]
Thus the total contribution of scale \(M\) is at most
\[
  O(M^3)e(\tilde{G}_M)=O(Mn^2).
\]
It remains to sum over the dyadic scales \(M=K,2K,4K,\ldots\) with \(M<L\). If the last such
scale is \(2^tK\), then \(2^tK<L\), and
\[
  K+2K+\cdots+2^tK
  =K(2^{t+1}-1)
  <2^{t+1}K
  <2L.
\]
Therefore \(\sum_{M<L}M=O(L)\), and the total contribution over all scales is \(\sum_MO(Mn^2)=O_K(Ln^2)\).
\end{proof}

We say that a pair \(\{u,v\}\) is \(K\)-thin if \(d(u,v)<K\). When \(K\) is fixed, we simply call it
thin. A copy of \(C_6\) is \(K\)-thin if each alternating triple contains a \(K\)-thin pair.

\begin{lemma}\label{lem:thin}
The number of \(K\)-thin copies of \(C_6\) is $O_K(n^{5/2})$.
\end{lemma}

\begin{proof}
For each \(K\)-thin \(C_6\) written as \(a\)-\(x\)-\(b\)-\(y\)-\(c\)-\(z\)-\(a\), inspect the pairs
\(ab,bc,ca\) in this order and choose the first \(K\)-thin one. Similarly, inspect \(xy,yz,zx\) in
this order and choose the first \(K\)-thin one. This assigns the cycle to exactly one of the
following nine placements:
\[
\begin{array}{c|ccc}
 & xy & yz & zx\\ \hline
ab & \mathrm{A} & \mathrm{O} & \mathrm{A}\\
bc & \mathrm{A} & \mathrm{A} & \mathrm{O}\\
ca & \mathrm{O} & \mathrm{A} & \mathrm{A}
\end{array}
\]
where \(\mathrm{O}\) means that the corresponding two-edge arcs are vertex-disjoint, and
\(\mathrm{A}\) means that they share one cycle-edge. Thus the opposite placements are precisely
\((ab,yz)\), \((bc,zx)\), and \((ca,xy)\); the other six are adjacent. 

Consider a \(K\)-thin copy $C=a$-$x$-$b$-$y$-$c$-$z$-$a$. If its chosen thin pairs form the opposite pattern
\((ab,yz)\), then \(d(a,b)<K\) and \(d(y,z)<K\). First choose the directed edges \(by\) and
\(za\). There are at most \((2e(G))^2\) choices. Once they are fixed,
\(x\in N(a,b)\) and \(c\in N(y,z)\) each have fewer than \(K\) choices. Thus, the three opposite patterns contribute \(O_K(e(G)^2)\).

For an adjacent pattern, by symmetry it is enough to consider the case $d(a,b)<K$ and $d(x,y)<K$,
where the two thin pairs correspond to the two arcs $a$-$x$-$b$ and $x$-$b$-$y$, which share the edge
\(xb\). Fix the 3-path $a$-$x$-$b$-$y$ and put
\[
  P=N(y)\setminus\{a,x,b,y\},\qquad
  Q=N(a)\setminus\{a,x,b,y\}.
\]
A completion $y$-$c$-$z$-$a$ corresponds to an ordered pair \((c,z)\in P\times Q\) with
\(cz\in E(G)\). Set $I=P\cap Q$. 
We partition the relevant edges among the four graphs
\[
  G[P\setminus Q,Q\setminus P],\qquad G[P\setminus Q,I],\qquad G[I,Q\setminus P],\qquad G[I].
\]
None of these graphs contains a 3-path $p_1$-$p_2$-$p_3$-$p_4$. In each of the first three
graphs, the path may be oriented so that \(p_1\in Q\) and \(p_4\in P\); in the fourth graph both
endpoints lie in \(I\subseteq P\cap Q\). In either case,
$a$-$p_1$-$p_2$-$p_3$-$p_4$-$y$-$b$-$x$-$a$
is a copy of \(C_8\), a contradiction. 

The Erd\H{o}s--Gallai theorem~\cite{ErdosGallai} gives at most \(N\)
edges in every \(N\)-vertex graph containing no 3-path. Hence the four graphs above have
at most \(4(|P|+|Q|)\) edges in total. An edge inside \(I\) represents two ordered pairs
\((c,z)\), and every other relevant edge represents one, so the number of completions is at most
\[
  8(|P|+|Q|)\le8(d(y)+d(a)).
\]

Therefore the number of \(C_6\) copies assigned to this adjacent pattern is at most
\[
  8\sum_{\substack{\text{paths }a\text{-}x\text{-}b\text{-}y\\ d(a,b)<K,\ d(x,y)<K}}
  (d(a)+d(y)).
\]
For each ordered pair \((a,b)\) with
\(d(a,b)<K\), there are \(d(a,b)\) choices for \(x\in N(a,b)\), and then at most \(d(b)\) choices
for \(y\in N(b)\). Hence
\[
  \sum_{\substack{\text{paths }a\text{-}x\text{-}b\text{-}y\\ d(a,b)<K,\ d(x,y)<K}}d(a)
  \le\sum_{\substack{a\ne b\\d(a,b)<K}}d(a)d(a,b)d(b)
  \le K\left(\sum_ad(a)\right)^2
  =4Ke(G)^2.
\]
The \(d(y)\) term is symmetric, now using the thin pair \(xy\). Multiplying by the six adjacent
placements and the preceding absolute completion constant still gives \(O_K(e(G)^2)\). Together
with the opposite placements, all thin cycles contribute \(O_K(e(G)^2)\). 
Lemma~\ref{lem:basic} shows that $e(G)=O(n^{5/4})$, which completes the proof.
\end{proof}

\begin{lemma}\label{lem:zero-fat}
For fixed sufficiently large \(K\), the total contribution of charged copies whose charged triple is 0-fat is
\[
  O_K(\rho n^3)+O_K(n^{5/2}).
\]
\end{lemma}

\begin{proof}
Let $C=a$-$x$-$b$-$y$-$c$-$z$-$a$ be charged to the zero-fat triple \(T=\{a,b,c\}\), and let \(W=\{x,y,z\}\). Since
\(T\) is zero-fat, all three pair-codegrees in \(T\) are below \(\rho n\), and hence
\(\mu(T)<\rho n\). Since the cycle is charged to \(T\),
\[
  \mu(W)\le\mu(T)<\rho n.
\]
Thus every pair-codegree in \(W\) is below \(3\rho n\), while every pair-codegree in \(T\) is
below \(\rho n\). 

If one of \(T,W\) has all three pair-codegrees at least \(K\), choose one such alternating triple
\(S\) only for this estimate: take \(S=T\) if \(T\) has this property, and otherwise take \(S=W\). For a fixed choice of
\(S=\{x_1,x_2,x_3\}\), the number of completions is at most $d(x_1,x_2)d(x_2,x_3)d(x_1,x_3)$. For
\(n\) large enough that \(3\rho n\ge K\), applying
Lemma~\ref{lem:all-large} with \(L=3\rho n\) gives \(O_K(\rho n^3)\) cycles; if
\(3\rho n<K\), this class is empty. If neither side has all three pair-codegrees at least \(K\), the cycle is \(K\)-thin, and
Lemma~\ref{lem:thin} applies.
\end{proof}

Combining these cases gives the desired estimate.

\begin{lemma}\label{lem:global}
For fixed sufficiently large \(K\) and fixed \(0<\rho<1\), every sufficiently large \(n\)-vertex
\(C_8\)-free graph satisfies
\[
  \N(C_6,G)\le n^3+O_K(\rho n^3)+O_K(n^{5/2})+O_{K,\rho}(n^2).
\]
\end{lemma}

\begin{proof}
Use equation~\eqref{eq:fat-triangles} for the 3-fat triples, Lemma~\ref{lem:one-two-fat} for the
2-fat and 1-fat triples, and Lemma~\ref{lem:zero-fat} for the zero-fat triples.
\end{proof}

Since \(K_{3,n-3}\) is \(C_8\)-free and contains \(6\binom{n-3}{3}=n^3+O(n^2)\) copies of
\(C_6\), while Lemma~\ref{lem:global} gives the matching upper bound after taking
\(\rho\to0\), we obtain
\[
  \ex(n,C_6,C_8)=(1+o(1))n^3.
\]
We shall need the corresponding stability statement. Put
\begin{equation}\label{eq:Rn-def}
  R(n):=6\binom{n-3}{3}+12(n-5)=n^3-12n^2+59n-120.
\end{equation}

\begin{lemma}\label{lem:stability}
For every \(\delta>0\), there is \(\rho>0\) such that the following holds for all sufficiently large
\(n\). If \(G\) is \(C_8\)-free and $\N(C_6,G)\ge R(n)$,
then some triple \(\{a,b,c\}\) satisfies
\[
  |N(a)\cap N(b)\cap N(c)|\ge(1-\delta)n.
\]
\end{lemma}

\begin{proof}
Choose once and for all \(K\ge10\), sufficiently large for Lemma~\ref{lem:all-large}. With that
\(K\) fixed, choose \(\rho>0\) sufficiently small
in terms of \(K\) and \(\delta\), and only then take \(n\) sufficiently large in terms of
\(K,\rho,\delta\). By Lemmas~\ref{lem:one-two-fat} and~\ref{lem:zero-fat}, all charged triples
that are not 3-fat contribute at most $O_K(\rho n^3)+O_K(n^{5/2})+O_{K,\rho}(n^2)$.

Using \(R(n)=n^3+O(n^2)\) and the assumption \(\N(C_6,G)\ge R(n)\), the 3-fat triples contribute at least
\begin{equation}\label{eq:three-fat-lower}
  \sum_{\substack{S\text{ 3-fat}}}F^\#(S)
  \ge (1-C_K\rho-o_{K,\rho}(1))n^3
\end{equation}
for a constant \(C_K\) depending only on \(K\).

For all sufficiently large \(n\), the lower bound in \eqref{eq:three-fat-lower} is positive, so \(\tilde{G}\) contains at least one
triangle. Let \(T_1,\ldots,T_s\) be the triangles of \(\tilde{G}\), and put \(\mu_i=\mu(T_i)\). By
equations~\eqref{eq:charge} and~\eqref{eq:three-fat-lower},
\begin{equation}\label{eq:mu-cube-lower}
  \sum_i\mu_i^3\ge(1-C_K\rho-o_{K,\rho}(1))n^3.
\end{equation}
On the other hand, Lemma~\ref{lem:packing} gives
\begin{equation}\label{eq:mu-sum-upper}
  \sum_i\mu_i\le n+O_\rho(1).
\end{equation}
Writing \(\mu_{\max}=\max_i\mu_i\), we have
\begin{equation}\label{eq:mu-cube-max}
  \sum_i\mu_i^3\le\mu_{\max}^2\sum_i\mu_i.
\end{equation}
Combining \eqref{eq:mu-cube-lower}, \eqref{eq:mu-sum-upper}, and \eqref{eq:mu-cube-max} gives
\begin{align*}
\mu_{\max}^2
  \ge\frac{(1-C_K\rho-o_{K,\rho}(1))n^3}{n+O_\rho(1)}
\quad\text{ and thus }\quad
  \mu_{\max}\ge(1-C'_K\rho-o_{K,\rho}(1))n.
\end{align*}
  
Choose \(T_i=\{a,b,c\}\) attaining the maximum. Then
\[
  |N(a,b)|+|N(b,c)|+|N(c,a)|\ge3(1-\varepsilon)n,
\]
where \(\varepsilon=C'_K\rho+o_{K,\rho}(1)\). Therefore, by the union bound for complements,
\begin{align*}
  |N(a,b)\cap N(b,c)\cap N(c,a)|
  \ge |N(a,b)|+|N(b,c)|+|N(c,a)|-2n
  \ge(1-3\varepsilon)n.
\end{align*}
This intersection is \(N(a)\cap N(b)\cap N(c)\). Choose \(\rho\) so that
\(3C'_K\rho<\delta/2\), and then take \(n\) large enough that the remaining \(o_{K,\rho}(1)\) term
is below \(\delta/6\). The result follows.
\end{proof}

\section{Defect absorption and exactness}

Throughout this section, \(G\) is an \(n\)-vertex \(C_8\)-free graph. We remove the exceptional
vertices left by Lemma~\ref{lem:stability}. Fix a triple \(\{a,b,c\}\subseteq V(G)\), put
\(A=\{a,b,c\}\), and define
\[
  M\coloneqq N(a)\cap N(b)\cap N(c),\qquad
  Q:=V(G)\setminus(A\cup M).
\]
Since neighborhoods are open, \(A\cap M=\varnothing\). Write \(m=|M|\) and \(q=|Q|\). 

\begin{lemma}\label{lem:terminal}
If \(m\ge5\), then
\[
  e(G[M])\le1
  \qquad\text{and}\qquad
  |N(v)\cap M|\le1\quad\text{for every }v\in Q.
\]
Moreover,
\[
  \N(C_6,G[A\cup M])=6\binom m3+4e(G[A])e(G[M])(m-2).
\]
\end{lemma}

\begin{proof}
First suppose that \(G[M]\) contains two distinct edges. There are two cases. If the two edges are
disjoint, write them as \(u_1u_2\) and \(u_3u_4\). Since \(m\ge5\), choose
\(u_5\in M\setminus\{u_1,u_2,u_3,u_4\}\). Then $u_1$-$u_2$-$a$-$u_3$-$u_4$-$b$-$u_5$-$c$-$u_1$
is a copy of \(C_8\). If the two edges share a vertex, write them as \(u_1u_2\) and \(u_1u_3\). Choose
distinct \(u_4,u_5\in M\setminus\{u_1,u_2,u_3\}\). Then $u_2$-$u_1$-$u_3$-$a$-$u_4$-$b$-$u_5$-$c$-$u_2$
is a copy of \(C_8\). Hence \(e(G[M])\le1\).
Likewise, if \(v\in Q\) has two neighbors \(u_1,u_2\in M\), choose distinct
\(u_3,u_4\in M\setminus\{u_1,u_2\}\). Then
  $v$-$u_1$-$a$-$u_3$-$b$-$u_4$-$c$-$u_2$-$v$
is a copy of \(C_8\). This proves the structural assertions.

Let \(C\) be a copy of \(C_6\) using \(t_C\) vertices of \(M\), and set \(a_C=6-t_C\). Let
\(k,\ell\), and \(h\) be the numbers of cycle-edges inside \(M\), inside \(A\), and across the
partition, respectively. Counting incidences with the vertices in the two parts gives
\[
  h+2k=2t_C,
  \qquad
  h+2\ell=2a_C.
\]
Subtracting and using \(a_C+t_C=6\) yields
\[
  t_C=3+\frac{k-\ell}{2}.
\]
Since at most three vertices lie in \(A\), we have \(t_C\ge3\), and hence \(k\ge\ell\). As
\(k\in\{0,1\}\) and \(k-\ell\) is even, the only possibilities are
\((k,\ell)\in\{(0,0),(1,1)\}\).

In the first case, the cycle alternates between all three vertices of \(A\) and a chosen 3-set of \(M\).
The resulting \(K_{3,3}\) has \(3!\,2!/2=6\) unoriented Hamilton cycles, contributing
\(6\binom m3\). In the second case, choose an edge \(ab\) of \(A\), the unique chosen edge \(xy\)
of \(M\), and a third vertex \(z\in M\setminus\{x,y\}\). If \(c\) is the third vertex of \(A\), then,
up to reversal, the four cycles are
\begin{align*}
  &a\text{-}b\text{-}x\text{-}y\text{-}c\text{-}z\text{-}a,\qquad a\text{-}b\text{-}y\text{-}x\text{-}c\text{-}z\text{-}a,\quad
  a\text{-}b\text{-}z\text{-}c\text{-}x\text{-}y\text{-}a,\qquad a\text{-}b\text{-}z\text{-}c\text{-}y\text{-}x\text{-}a.
\end{align*}
Thus this case contributes \(4e(G[A])e(G[M])(m-2)\), proving the formula.
\end{proof}

The next two lemmas bound cycles using vertices of \(Q\).

\begin{lemma}\label{lem:one-defect}
Assume \(m\ge5\) and fix \(v\in Q\). The number of copies of \(C_6\) containing \(v\) and no
other vertex of \(Q\) is at most
\[
  m(m-1)+300m.
\]
\end{lemma}

\begin{proof}
Put \(B=M\cup\{v\}\). By Lemma~\ref{lem:terminal}, \(G[M]\) has at most one edge and \(v\) has
at most one neighbor in \(M\); hence \(e(G[B])\le2\). For a cycle \(C\), let
\(b_C=|V(C)\cap B|\), and let \(k_C,\ell_C\) be the numbers of cycle-edges inside \(B\) and \(A\),
respectively. As above,
\begin{align}\label{eq:bc-value}
    b_C=3+\frac{k_C-\ell_C}{2}.
\end{align}
Since \(C\) contains \(v\), has no other vertex of \(Q\), and uses at most three vertices of \(A\),
we have \(b_C\ge3\). Since \(k_C\le2\), we also have \(b_C\le4\). The possible triples are therefore
exactly
\[
  (b_C,k_C,\ell_C)\in
  \{(3,0,0),(3,1,1),(3,2,2),(4,2,0)\}.
\]

In the case \((3,0,0)\), $C$ uses three vertices from \(B\) and all three vertices of \(A\), and it has no internal edge in either
part. Since $C$ contains \(v\), the three vertices from \(B\) are \(v,u_1,u_2\) for some
pair \(\{u_1,u_2\}\subseteq M\), and the cycle must alternate between \(A\) and
\(\{v,u_1,u_2\}\). Since \(v\notin M=N(a)\cap N(b)\cap N(c)\), it misses at least one vertex of
\(A\), and hence has two neighbors in \(A\). After the pair \(\{u_1,u_2\}\) is fixed, a
cycle can exist only if the two cycle-neighbors of \(v\) are its two available neighbors in \(A\).
The two vertices \(u_1,u_2\) can then be interchanged, so this gives at most two cycles for each
choice of \(\{u_1,u_2\}\). Thus this type contributes $2\binom m2=m(m-1)$.

We now give explicit bounds for the remaining three types. On any fixed six-vertex set there are at
most \(6!/(6\cdot 2)=60\) unoriented cyclic orders, so 60 is a universal upper bound for the number of copies of
\(C_6\) on that set.

For type \((3,1,1)\), the unique internal \(B\)-edge is either the possible edge of \(G[M]\), or the
possible edge from \(v\) to \(M\). Suppose first that the internal \(B\)-edge is the edge \(xy\) of
\(G[M]\). Since \(C\) contains \(v\) and has exactly three vertices in \(B\), the three \(B\)-vertices
are forced to be \(\{v,x,y\}\).
Suppose instead that the internal \(B\)-edge is the edge \(vx\) from \(v\) to \(M\). Then the third
\(B\)-vertex may be chosen from \(M\setminus\{x\}\). Hence there are
at most \(1+(m-1)=m\) possible six-vertex sets and at most \(60m\) cycles of this type.

For type \((3,2,2)\), we have \(k_C=2\). By Lemma~\ref{lem:terminal}, \(G[M]\) has at most one edge
and \(v\) has at most one neighbor in \(M\). 
Write these two edges as \(xy\in E(G[M])\) and \(zv\), where \(z\in M\). If
\(z=x\) or \(z=y\), then the two internal \(B\)-edges have exactly three endpoints, and hence the
\(B\)-vertex set is determined. If \(z\notin\{x,y\}\), then the two edges \(xy\) and \(zv\) have four
endpoints in \(B\), contradicting \(b_C=3\). Thus, whenever $(3,2,2)$ is feasible, the \(B\)-vertex set
is determined; since \(b_C=3\), all three vertices of \(A\) are also used. Hence the six-vertex set is
determined, giving at most \(60\) cycles.

Finally, consider type \((4,2,0)\).  Write these two $B$-edges as \(xy\in E(G[M])\) and \(zv\), where
\(z\in M\). If \(z=x\) or \(z=y\), then the two internal \(B\)-edges have exactly three endpoints in
\(B\). Since \(b_C=4\), the fourth \(B\)-vertex may be chosen from \(M\setminus\{x,y\}\), giving at
most \(m-2\) choices. If \(z\notin\{x,y\}\), then the two edges \(xy\) and \(zv\) already have four
endpoints in \(B\), so the \(B\)-vertex set is determined. Hence there are at most
\((m-2)+1=m-1\) choices for the \(B\)-vertex set. There are three choices for the two vertices used
from \(A\), and therefore at most \(\binom32\cdot60(m-1)=180(m-1)\) cycles of this type. For \(m\ge5\),
\[
  60m+60+180(m-1)\le300m.
\]
Adding the main type proves the stated bound.
\end{proof}

\begin{lemma}\label{lem:several-defects}
Assume \(m\ge5\). The number of copies of \(C_6\) containing at least two vertices of \(Q\) is
$
  O(q^2m+q^3),
$ 
with an absolute implicit constant.
\end{lemma}

\begin{proof}
Recall that \(Q=V(G)\setminus(A\cup M)\) and \(q=|Q|\). If \(q<2\), there is no cycle containing at
least two vertices of \(Q\), so the assertion is immediate. Assume henceforth that \(q\ge2\).

Let \(M_1\subseteq M\) be the set of vertices with a neighbor in \(Q\), and let \(M_2\) be the set of
endpoints of the possible edge in \(G[M]\). By Lemma~\ref{lem:terminal},
\[
  |M_1|\le q,
  \qquad
  |M_2|\le2.
\]
Put
\[
  m_1=|M\setminus(M_1\cup M_2)|,
  \qquad
  G'\coloneqq G[A\cup Q\cup M_1\cup M_2].
\]
First, we claim that a copy of \(C_6\) containing at least two vertices of \(Q\) contains at most one
vertex of \(M\setminus(M_1\cup M_2)\). Suppose otherwise, and let \(C\) contain distinct
\(u,u'\in M\setminus(M_1\cup M_2)\). By the definition of \(M_1\), neither \(u\) nor \(u'\) has a
neighbor in \(Q\); by the definition of \(M_2\), neither is incident to an edge of \(G[M]\). Since
\(u,u'\in M\), both are adjacent to every vertex of \(A\). Hence the two cycle-neighbors of each of
\(u\) and \(u'\) must lie in \(A\).

The cycle already contains \(u,u'\) and at least two vertices of \(Q\). If it used at most one vertex
of \(A\), then \(u\) could not have two distinct cycle-neighbors. Thus the only possible case is that
\(C\) uses exactly two vertices \(v,v'\in A\) and exactly two vertices \(w,w'\in Q\). Then both
\(u\) and \(u'\) must have cycle-neighborhood \(\{v,v'\}\), forcing the four cycle-edges
\(v u,uv',v'u',u'v\). These four edges form the cycle \(v\)-\(u\)-\(v'\)-\(u'\)-\(v\), so
\(v,u,v',u'\) already have degree two in the chosen cycle-edge set and cannot be connected to
\(w,w'\). This contradicts the assumption that \(C\) is a copy of \(C_6\).

Since \(q\ge2\) and \(|M_1\cup M_2|\le q+2\), the graph \(G'\) has \(O(q)\) vertices. If the cycle avoids
\(M\setminus(M_1\cup M_2)\), it lies in \(G'\), so Lemma~\ref{lem:basic} gives \(\mathrm{N}(C_6,G')=O(q^3)\).

Suppose the cycle contains exactly one vertex \(u\in M\setminus(M_1\cup M_2)\). Its two cycle-neighbors
are distinct vertices \(v,v'\in A\). After deleting \(u\), the remaining five vertices form a copy of
\(P_5\), namely a four-edge \(v\)--\(v'\) path in \(G'\). If no such path exists, there is nothing to
count. Otherwise fix one such path and let \(S_0\) be its internal set, so \(|S_0|=3\). Any other
four-edge \(v\)--\(v'\) path in \(G'\) must meet \(S_0\) internally, since two internally disjoint
four-edge \(v\)--\(v'\) paths form a copy of \(C_8\). For each \(x\in S_0\), after choosing the position of
\(x\) on the path, the other two internal vertices can be chosen in at most \(|V(G')|^2\) ordered ways.
Hence there are \(O(|V(G')|^2)=O(q^2)\) such paths for each ordered pair \((v,v')\), and there are at
most \(m_1\) choices for \(u\). This gives \(O(m_1q^2)\le O(mq^2)\) cycles.
\end{proof}

We now compare the upper bound for \(\N(C_6,G)\) obtained from the preceding lemmas with the benchmark value \(R(n)\) defined in \eqref{eq:Rn-def}. The aim is to show that if \(Q\) is nonempty but still small compared with \(M\), then \(\N(C_6,G)<R(n)\).

\begin{lemma}\label{lem:absorption}
There are absolute constants \(\delta_0>0\) and \(m_0\ge5\) such that the following holds. If
$m\ge m_0,0<q\le\delta_0m$, 
then
\[
  \N(C_6,G)<R(n).
\]
\end{lemma}

\begin{proof}
By Lemmas~\ref{lem:terminal}--\ref{lem:several-defects}, for an absolute constant \(C\),
\[
  \N(C_6,G)
  \le6\binom m3+12(m-2)+qm(m-1)
     +C(qm+q^2m+q^3).
\]
The explicit \(300qm\) term from Lemma~\ref{lem:one-defect} is included in the error term \(C(qm+q^2m+q^3)\).
Since \(n=3+m+q\),
\[
  R(n)=6\binom{m+q}{3}+12(m+q-2).
\]
Using
\[
  6\left(\binom{m+q}{3}-\binom m3\right)
  =3qm(m-1)+3mq(q-1)+q(q-1)(q-2),
\]
we obtain, for \(q\ge1\),
\begin{align*}
  R(n)-\N(C_6,G)
  &\ge2qm(m-1)+3mq(q-1)+q(q-1)(q-2)+12q-C(qm+q^2m+q^3)\\
  &\ge2qm(m-1)-C(qm+q^2m+q^3).
\end{align*}
Choose \(0<\delta_0\le1\) and \(m_0\ge5\) large enough
so that for every $m\ge m_0$, 
\[
  C(\delta_0+\delta_0^2)<\frac14\quad\text{ and }\quad 2m(m-1)-Cm\ge\frac32m^2.
\]

Since $q^2m\le\delta_0qm^2$ and $q^3\le\delta_0^2qm^2$,
we obtain
\[
  R(n)-\N(C_6,G)
  \ge\frac32qm^2-C(\delta_0+\delta_0^2)qm^2>0.
\]
\end{proof}

\begin{proof}[Proof of Theorem~\ref{thm:main}]
As observed in the introduction, the graph
$
  H_n=K_3\vee(K_2\cup I_{n-5})
$ 
is \(C_8\)-free; the count in Lemma~\ref{lem:terminal} gives \(\N(C_6,H_n)=R(n)\). Hence
\[
  \ex(n,C_6,C_8)\ge R(n).
\]

Let \(G\) be an extremal \(n\)-vertex \(C_8\)-free graph. Let \(\delta_0,m_0\) be given by
Lemma~\ref{lem:absorption}, decreasing \(\delta_0\) if necessary so that \(\delta_0\le1\). Apply
Lemma~\ref{lem:stability} with $\delta=\delta_0/10.$ 
We obtain a triple \(A=\{a,b,c\}\) such that, with \(M:=N(a)\cap N(b)\cap N(c)\), we have \(m:=|M|\ge(1-\delta)n\). Put \(Q:=V(G)\setminus(A\cup M)\) and \(q:=|Q|\). Then
\[
  q=n-3-m\le n-m\le\delta n.
\]
Consequently,
\[
  \frac qm\le\frac{\delta}{1-\delta}
  =\frac{\delta_0/10}{1-\delta_0/10}<\delta_0,
\]
where we used \(\delta_0\le1\). For sufficiently large \(n\) we also have \(m\ge m_0\). If \(q>0\),
then Lemma~\ref{lem:absorption} gives \(\N(C_6,G)<R(n)\), a contradiction. Thus
\(Q=\varnothing\).

Now \(V(G)=A\mathbin{\dot\cup}M\), every edge between \(A\) and \(M\) is present, and
Lemma~\ref{lem:terminal} gives \(e(G[M])\le1\). Put
\[
  r=e(G[A]),
  \qquad
  s=e(G[M]).
\]
Then
\[
  \N(C_6,G)=6\binom{n-3}{3}+4rs(n-5)\le R(n).
\]
Equality is necessary. Since \(0\le r\le3\), \(s\in\{0,1\}\), and \(n-5>0\), equality forces
\(rs=3\); hence \(r=3\) and \(s=1\). Therefore
\[
  G\cong K_3\vee(K_2\cup I_{n-5}).
\]
This proves both the exact value and uniqueness.
\end{proof}

\section*{Declaration on the use of AI}

The authors used generative AI tools to assist in discussing proof strategies, checking proofs, and improving exposition.

\end{document}